\documentclass[12pt,letterpaper,reqno]{amsart}
\usepackage[polutonikogreek, english]{babel}
\usepackage{enumerate}

\usepackage[all]{xy}
\newdir{|>}{!/4.5pt/\dir{|}
*:(1,-.2)\dir^{>}
*:(1,+.2)\dir_{>}}
\newdir{(>}{{}*!/-8pt/\dir{>}}
\newdir^{((}{{}*!/-5pt/\dir^{(}}
\xyoption{2cell}
\xyoption{curve}
\xyoption{rotate}
\UseAllTwocells

\usepackage{bbm}
\usepackage{mathrsfs}
\usepackage{amsfonts}
\usepackage{bm}
\usepackage{amsmath,amssymb,amsthm,stmaryrd}

\usepackage{hyperref}

\usepackage{bussproofs}

\usepackage[numbers, sort&compress]{natbib}
\usepackage{mathbbol}

\swapnumbers

\newcommand{\nocontentsline}[3]{}
\newcommand{\tocless}[2]{\bgroup\let\addcontentsline=\nocontentsline#1*{#2}\egroup}

\newtheorem{teorema}{Theorem}[section]

\newtheorem*{teoA}{Theorem A}
\newtheorem*{teoB}{Theorem B}
\newtheorem*{teoC}{Theorem C}

\newtheorem{cor}[teorema]{Corollary}

\newtheorem{proposicion}[teorema]{Proposition}
\newtheorem{prop}[teorema]{Proposition}

\theoremstyle{definition}

\newtheorem{defi}[teorema]{Definition}

\newcommand{\E}{\mathcal{E}}

\newcommand{\Sub}{\mathrm{Sub}}

\DeclareMathOperator{\dec}{Dec}

\newcommand{\set}{\mathbf{Set}}

\begin{document}
\renewcommand{\abstractname}{\vspace{-\baselineskip}}

\title[Connectedness through Decidable Quotients]{Connectedness through Decidable Quotients}
\author[E. Ruiz-Hern\'andez]{Enrique Ruiz-Hern\'andez}
\address[E. Ruiz-Hern\'andez]{Centro de Investigaci\'on en Teor\'ia de Categor\'ias y sus Aplicaciones, A.C. M\'e\-xi\-co.}
\curraddr{}
\email{e.ruiz-hernandez@cinvcat.org.mx}

\author[P. Sol\'orzano]{Pedro Sol\'orzano}
\address[P. Sol\'orzano]{Instituto de Matem\'aticas, Universidad Nacional Aut\'onoma de M\'exico, Oaxaca de Ju\'a\-rez, M\'e\-xi\-co.}
\curraddr{}
\email{pedro.solorzano@matem.unam.mx}
\date{\today}
\dedicatory{In memory of William Lawvere}
\begin{abstract}
{\bf R\'esum\'e.} 
En consid\'erant des quotients d\'ecidables, on fournit une condition suffisante (1) pour garantir que la sous-cat\'egorie pleine des objets d\'ecidables d'un topos soit un id\'eal exponentiel et (2) pour que la notion classique de connexit\'e pour un objet X co\"incide avec $\Pi X=1$, o\`u $\Pi$ est le foncteur adjoint \`a gauche de l'inclusion des d\'ecidables.\\

L'ajout de cette condition-ci dans le contexte de l'axiomatique de McLarty pour la G\'eom\'etrie Diff\'erentielle Synth\'etique rend tout topos qui la satisfait pr\'ecoh\'esif sur le topos de ses objets d\'ecidables. Une r\'eciproque est \'egalement fournie.\\

{\bf Abstract.} 
By looking at decidable quotients, a sufficient condition is provided to guarantee that (1) the full subcategory of decidable objects of a topos is an exponential ideal and that (2) the classical notion of connectedness for an object $X$ coincides with $\Pi X=1$, where $\Pi$ is the left-adjoint functor of the inclusion of the decidable objects.\\

The addition of this condition to McLarty's axiomatic set up for Synthetic Differential Geometry makes any topos that satisfies it precohesive over the topos of its decidable objects. A converse is also provided. \\
\end{abstract}

\maketitle

\newpage
\section*{Motivation}

Colin McLarty \cite{MR0925615} formalizes a development of the notion of set out of that of space through topos theory: he considers a topos of \emph{spaces} from which it is possible to get a category of sets, a topos of sets. In that paper, he posits two axiom systems, one for the topos of sets and another for the topos of spaces, which he denotes $\mathrm{SDG}$---for several of the postulates pertain specifically to the topic of Synthetic Differential Geometry. There are two postulates in this latter system, $\mathrm{SDG}_6$ and $\mathrm{SDG}_7$ which he further studies in \cite{MR0877866}. There they are presented as
\begin{center}
\begin{minipage}{8.5cm}
\hspace{-.8cm}(NS) Any object is either initial or has a global element, 
\end{minipage}
\end{center}
and
\begin{center}
\begin{minipage}{8.5cm}
\hspace{-1.3cm}(DSO) There exists a unique decidable subobject for any given object containing all of its points.
\end{minipage}
\end{center}
In McLarty's words, intuitively, the Nullstellensatz (NS) ``says points are the smallest spaces, so two points in any space are either wholly coincident or wholly disjoint''; that is, ``for every space $B$ and points $b_1\in B$ and $b_2\in B$, `$b_1=b_2\vee\neg(b_1=b_2)$' is true even if $B$ is not discrete (decidable) and the corresponding sentence with variables over $B$ is not true'' (\cite[p. 81]{MR0925615}). 

Toposes satisfying NS abound: Any topos $\E$ precohesive over a topos $\mathcal S$ that satisfies NS must also satisfy NS---thus any topos that is precohesive over $\set$ satisfies NS. Indeed, by Lawvere's Nullstellensatz, $f_!(X)$ is initial if $f_*(X)$ is initial. In such a case, by the strictness of $0$, $X$ would be initial too.  Hence $f_*(X)$ is not initial whenever $X$ is not initial either. 

These two axioms imply that the topos of spaces we begin with has a local geometric morphism to a category of---decidable---sets, as proved in \cite{MR0877866}. There is just one functor missing to aspire to have precohesion in the sense of William Lawvere's program \cite{MR2369017}.  
\section*{Main Results}
The purpose of this communication is to ultimately focus in such a functor, first in isolation, and then in the precohesive context.  To this effect, proceeding \`a la McLarty, consider the following postulate:
\begin{center}
\begin{minipage}{9.5cm}
\hspace{-1.3cm}(DQO) There exists a unique decidable quotient for any given object along which any arrow to $2$ factors uniquely. 
\end{minipage}
\end{center}
In the presence of DQO, let $p_X:X\to \Pi(X)$ be the corresponding quotient map. It follows that $\Pi 1\cong 1$, $\Pi X=0$ if and only if $X=0$, and  $\Pi\Pi X\cong\Pi X$ for every object $X\in\E$.  

It is not immediately apparent that DQO implies that $\Pi$ is indeed the object part of a left adjoint of the inclusion of $\dec(\E)$ in $\E$.  This is actually the case, as established by the following result. 

\begin{teoA}\hypertarget{T:TeorA} For a topos $\E$ that satisfies  {\em NS}  and {\em DQO}, $\Pi$ is the object part of a finite-product-preserving functor left adjoint to the inclusion of $\dec(\mathcal E)$.\end{teoA}

Conversely, by \ref{T:DecReflexImpliesDQO}, DQO holds in the presence of NS as soon as the inclusion $\dec(\E)$ into $\E$ has a left adjoint. Therefore,
 
\begin{teoB}\hypertarget{T:TeorB} For a topos $\E$ that satisfies  {\em NS} , $\dec(\E)$ is an exponential ideal as soon as it is reflective.
\end{teoB}

It is not immediately apparent that a topos $\E$ precohesive over $\dec(\E)$ satisfies McLarty's DSO. Section \ref{S:preco} is devoted to establishing this. It should be noted that the proof requires DQO.  Hence, 
\begin{teoC}\hypertarget{T:TeorC} Let $\mathcal E$ be a topos satisfying NS. Then $\E$ satisfies {\em DQO} and McLarty's {\em DSO} if and only if $\E$ is precohesive over $\dec(\E)$. 
\end{teoC}
In an extensive category, an object is connected if it has exactly two complemented subobjects. On the other hand,  \citet{MR2369017} calls an object connected in the context of precohesion if its image under the left-most adjoint is terminal. In this report, for an object $X$, the latter corresponds in a weak sense to $\Pi(X)=1$.  In fact, in the presence of NS and DQO both agree (see \ref{P:connect}).  Notice that to require DQO but not NS is not be enough:  In $\set\times\set$, DQO holds and thus $\Pi(1)=1$, yet $1$ has four complemented subobjects. 

In view of these observations, the addition of DQO is a natural extension of Mc\-Larty's axioms for SDG that frames it in a context of Lawvere's precohesion in which these two notions of connectedness agree. 

\pagebreak

\subsection*{Aknowledgements.} 
Special thanks to Mat\'ias Menni for several comments and remarks greatly improving the presentation of this report.  Also, to the couple of referees that helped the authors clarify their goals. Lastly, to Lilli\'an Bustamante and her staff for their hospitality the day some of the main ideas for this project were conceived. The second named author is supported by the SECIHTI Investigadoras e Investigadores por M\'exico Program Project No. 61.


\section{Connectedness}
An object $X$ is {\em connected} if it has exactly two complemented subobjects.  Let $\Sub_c(X)$ be the collection of complemented subobjects of $X$. These are evidently classified by $2$. The DQO axiom requires the complemented subobjects of $X$ to be in bijective correspondence with those of its decidable quotient $\Pi X$. In fact, 

\begin{proposicion}\label{P:connect} Let $\E$ satisfy NS and DQO. For an object $X$, $\Pi X=1$ if and only if  $\Sub_c(X)$ has exactly two elements, i.e. if $X$ is connected. 
\end{proposicion}
\proof For the necessity, since $\Pi X=1$ and NS guarantees that $\Sub_c(1)=2$, there are exactly two arrows $X\rightarrow 2$. So, as $X,0\in\Sub_c(X)$, these are all of the complemented subobjects of $X$.  

For the sufficiency, since $X\neq 0$, there is an arrow $1\rightarrow X$. Now, by considering the composite $1\to X\to 1$, it follows that $!_X:X\rightarrow 1$ is epic. Assuming that $\Sub_c(X)$ has exactly two elements, the two arrows $X\rightarrow 2$ are the constants, which factor through $!_X$; thus by DQO, the arrow $\Pi X\rightarrow 1$ is an isomorphism.
\endproof

\begin{proposicion}[Schanuel]
Let $\E$ satisfy NS and DQO.  The finite product of connected objects is connected.
\end{proposicion}
\begin{proof}
The argument syntactically coincides with that of \cite[Theorem 12.1.1]{RRZ2004}, as expected since the Nullstellensatz therein coincides with NS.  

Let $Z\rightarrowtail X\times Y$ be complemented in $X\times Y$ and different from $\varnothing$ and $X\times Y$. By NS, there are points $\langle a,b\rangle:1\to X\times Y$ and $\langle c,d\rangle:1\to X\times Y$ that factor through $Z$ and $Z^c$ respectively.

\newpage

The decomposition given by $Z$ induces a decomposition of $X$ via the map $\langle 1,b!_X\rangle:X\to X\times Y$ and a decomposition of $Y$ via the map $\langle c!_Y,1\rangle:Y\to X\times Y.$ Now, since $X$ and $Y$ are connected, these decompositions are trivial.   Therefore, $\langle 1,b!_X\rangle$ factors through $Z$ and $\langle c!_Y,1\rangle$ factors through $Z^c$.  This implies that $\langle b,c\rangle:1\to X\times Y$ factors through $Z\cap Z^c$. A contradiction. 

Therefore there does not exist a nontrivial complemented subobject of $X\times Y$, and thus the conclusion follows. 
\end{proof}

Before the next definition, let $P_c(X)$ be the subobject of the power object $P(X)$,
\begin{equation}
\{u\in P(X): u\cup u^c=X\}\rightarrowtail P(X),
\end{equation}
 where $u^c$ stands for $\{x\in X: x\notin u\}$. Observe that $\Sub_c(X)$ is in one-to-one correspondence with $\E(1, P_c(X))$.

Also, recall that in a topos, morphisms can be described as in set theory via their graphs: appropriate subobjects of the product of their domain and codomain (See Exercise VI.11 in the textbook by Mac Lane and Moerdijk \cite{MR1300636}): the subobject
\[
G\rightarrowtail X\times Y
\]
is the graph of an arrow $X\to Y$ if and only if 
\begin{equation}\label{E:graph}
\exists!y(\langle x, y\rangle\in G
\end{equation}
is universally valid for $x\in X$. In particular, for an arrow $f:X\to Y$ one writes $\lvert f\rvert\rightarrowtail X\times Y$ for its graph and $f^{-1}(y)$ instead of $\{x\in X: \langle x, y\rangle\in\lvert f\rvert\}$ for its standard fiber. 
\begin{defi} A map $f:X\to Y$ has  {\em pneumoconnected fibers} if the formula 
\begin{equation}\label{D:pneumofib}
\neg\neg(f^{-1}(y)\cap w=\varnothing\vee f^{-1}(y)\cap w^c=\varnothing) 
\end{equation}
is universally valid, with $y\in Y$ and $w\in P_c(X)$.
\end{defi}
Intuitively, it says that any generic fiber $f^{-1}(y)$ is very close to being connected: Except for the double negation, it reads that fibers cannot be separated through complemented objects.  One cannot rid oneself from the double negation in the definition, since $\neg\neg\alpha\Rightarrow\alpha$ is not universally valid in the internal logic of a non-boolean topos.  Yet to prove an assertion of the form $\neg \alpha$ (e.g. in the previous or in the following two propositions), a \linebreak 

\noindent perfectly valid intuitionistic argument (so long as one refrains from invoking the axiom of choice and the excluded middle) is to assume $\alpha$ and arrive at a contradiction $\perp$, since $\neg\alpha$ is equivalent to $\alpha\Rightarrow\perp$. 

When considering global elements \eqref{D:pneumofib} does capture the connectendess of fibers, as seen in the next few results. 
\begin{proposicion}
Let $\E$ satisfy NS and DQO.  And let $f:X\to Y$ have pneumoconnected fibers. For any point $b:1\to Y$, its fiber $f^{-1}(b)$, given by the pullback diagram

\[\xymatrix{
f^{-1}(b)\ar[r]\ar[d] & X\ar[d]^f \\
1\ar[r]_(.4)b & Y,
}\]
does not have nontrivial complemented subobjects, i.e. $\Pi(f^{-1}(b))=1$.
\end{proposicion}
\proof Let $A\rightarrowtail f^{-1}(b)$ such that $f^{-1}(b)=A+A^c$. If this is a nontrivial decomposition, then both $A$ and $A^c$ are not initial. By NS, there are points $a:1\to A$ and $a':1\to A^c$ such that $f\circ a=f\circ a'=b$. Therefore, $f^{-1}(b)\cap A\neq\varnothing$ and $f^{-1}(b)\cap A^c\neq\varnothing$, which proves that

\[
\neg (f^{-1}(b)\cap A=\varnothing\vee f^{-1}(b)\cap A^c=\varnothing),
\]
which contradicts \eqref{D:pneumofib}.
\endproof

\begin{proposicion}\label{P:pneumoprodfib}
Let $\E$ satisfy NS and DQO. Let $f:X\to Y$ and $g:X'\to Y'$ be two arrows with pneumoconnected fibers.  Then $f\times g$ has pneumoconnected fibers. 
\end{proposicion}
\proof
Define
\[
\vartheta:=(f\times g)^{-1}(\langle z,w\rangle)\cap v\neq\varnothing\wedge(f\times g)^{-1}(\langle z,w\rangle)\cap v^c\neq\varnothing
\]
and
\[
R:=\{\langle z,w\rangle\in Y\times Y':\exists v\in P_c(X\times X').\vartheta\}.
\]
Suppose for contradiction that $R$ is not initial. By the NS, there exist points $a:1\to Y$,  $b:1\to Y'$ and a complemented $D\rightarrowtail X\times X'$ such that
\begin{equation}\label{E:pneumoprod}
(f\times g)^{-1}(\langle a,b\rangle)\cap D\neq\varnothing\wedge(f\times g)^{-1}(\langle a,b\rangle)\cap D^c\neq\varnothing.
\end{equation}

On the other hand, $\Pi(f^{-1}(a))=1=\Pi(g^{-1}(b))$ and thus
\[
(f\times g)^{-1}(\langle a,b\rangle)=f^{-1}(a)\times g^{-1}(b)
\]
is also connected. Therefore, $(f\times g)^{-1}(\langle a,b\rangle)\cap D=\varnothing$ or $(f\times g)^{-1}(\langle a,b\rangle)\cap D^c=\varnothing$, which is a contradiction to \eqref{E:pneumoprod}.  Therefore, $R$ cannot have points and by NS it must be initial. That is that

\[
\neg\exists v\in P_c(X\times X').\vartheta
\]
is universally valid for $\langle z,w\rangle\in Y\times Y'$. Or, equivalently,
\[
\forall v\in P_c(X\times X').\neg\vartheta
\]
is universally valid for $\langle z,w\rangle\in Y\times Y'$.

Or, equivalently,
\[
\neg\neg((f\times g)^{-1}(\langle z,w\rangle)\cap v=\varnothing\vee(f\times g)^{-1}(\langle z,w\rangle)\cap v^c=\varnothing)
\]
is universally valid, with $z\in Y$, $w\in Y'$ and $v\in P_c(X\times X')$, as required by \eqref{D:pneumofib}. 
\endproof
Lastly, another result that can be proved along the same lines is the following. 
\begin{prop}\label{L:pullbackofmalongpX}
Let $\E$  satisfy NS and DQO. Any pullback of an arrow that has pneumoconnected fibers has pneumoconnected fibers.\end{prop}
\begin{proof} Straightforward.
\end{proof}


\section{Fiber Pneumoconnectedness Lemma}

 The purpose of this section is to state and prove the following result. It is central to several arguments in this report. 
 \begin{teorema}[Fiber Pneumoconnectedness Lemma]\label{T:Lemon}
Let $\E$ be a topos that satisfies NS. Let $q:X\twoheadrightarrow Q$ be epic. Then the following statements are equivalent:
\begin{enumerate}[(i)]
 \item\label{I:decquotfact2} Every arrow $X\rightarrow 2$ factors through $q$. \pagebreak
 \item\label{I:PCfibers} The map $q$ has pneumoconnected fibers.
 \item\label{I:decquotfactdec} Every arrow $X\rightarrow Y$ with $Y$ decidable factors through $q$.
\end{enumerate}
\end{teorema}
The rest of the section is devoted to some of its applications.  Afterwards, a proof of \ref{T:Lemon} will be given towards the \hyperlink{LemonProof}{end of the section}. 

\begin{cor}\label{T:Pipreservesproducts}
Le $\E$ satisfy NS and DQO and let $p_X:X\to \Pi(X)$ be the corresponding quotient map. Then $\Pi(X\times Y)\cong\Pi X\times\Pi Y$.
\end{cor}
\begin{proof}
By \ref{T:Lemon} both projections $p_X$ and $p_Y$ have pneumoconnected fibers, hence by \ref{P:pneumoprodfib} so does the epic arrow $p_X\times p_Y:X\times Y\rightarrow\Pi X\times\Pi Y$. By DQO, since $\Pi X\times\Pi Y$ is a decidable quotient that factors arrows to $2$, it coincides with $\Pi (X\times Y)$.
\end{proof}

\proof[Proof of \hyperlink{T:TeorA}{Theorem A}] For an arrow $f:X\rightarrow Y$ with $Y\in\dec(\E)$, by \ref{T:Lemon}\eqref{I:decquotfactdec} there exists a unique $f':\Pi X\rightarrow Y$ such that $f'\circ p_X=f$. So $\Pi$ is functorial and $\Pi\dashv \mathcal I$. By \ref{T:Pipreservesproducts} it also preserves products. 
\endproof

\begin{cor}\label{T:DecReflexImpliesDQO}
Let $\E$ satisfy NS with $\dec(\E)$ reflective. Then DQO holds.
\end{cor}
\begin{proof}
Let $\Pi\dashv\mathcal I$ be the reflection. Let $p:1\rightarrow \mathcal I \Pi$ be its unit. Since every arrow $X\rightarrow 2$ in $\E$ factors through $p_X$ and thus also through its image. Since the image is also decidable, then it is universal and hence the unit of the adjunction. Whence $p_X$ is epic. Thus there exists a decidable quotient that factors arrows to $2$. By \ref{T:Lemon}, $p_X$ has pneumoconnected fibers. 

To verify uniqueness, let $q:X\twoheadrightarrow Q$ and $q':X\twoheadrightarrow Q'$ with $Q$ and $Q'$ decidable be two quotients satisfying the factorization property of DQO.  Then, by \eqref{I:decquotfactdec}, there are arrows $Q\rightarrow Q'$ and $Q'\rightarrow Q$ which are necessarily inverses of each other. Thus one verifies DQO.
\end{proof}
In his context, McLarty \cite{MR0877866} proves that $\dec(\E)$ is actually a topos (see \citet{MR3893295} for some generalizations). 

The following result also invokes \ref{T:Lemon} and provides necessary and sufficient conditions for $\dec(\E)$ to be a topos. 
\begin{cor}\label{T:TeorD}  
Let $\E$ be a nondegenerate topos satisfying NS and {\em DQO}. The category $\dec(\E)$ is a topos if and only if the arrow $\Pi(f)$ is epic for every $\neg\neg$-dense arrow $f$.
\end{cor}
\pagebreak
\proof
Proposition VI.1 in \cite{MR1300636} establishes several equivalences for a topos to be boolean. Among which is that every subobject is complemented. Also, that the operator $\neg\neg$ is the identity, i.e. there are no nontrivial dense subobjects. 

Suppose that $\Pi(f)$ is epic for every $\neg\neg$-dense arrow $f:A\rightarrow X$. Now, let $m:B\rightarrowtail\Pi X$ be a monic arrow. Consider the following pullback diagrams:
\[\xymatrix{
R\ar@{(>->}[r]^(.55){p_X^{-1}(m)}\ar@{->>}[d] & X\ar@{->>}[d]^{p_X} &  P\ar@{(>->}[l]_{p_X^{-1}(m^c)}\ar@{->>}[d] \\
B\ar@{(>->}[r]_m & \Pi X &  B^c\ar@{(>->}[l]^(.45){m^c}
}\]
Since inverse images preserve pseudocomplements, $p_X^{-1}(m^c)\cong p_X^{-1}(m)^c$, without loss of generality $P=R^c$, and, by \ref{L:pullbackofmalongpX}, $B\cong\Pi R$ and $B^c\cong\Pi R^c$.

Now, as $r:R+R^c\rightarrowtail X$ is $\neg\neg$-dense, $\Pi(r)$ is epic. Consider the following commutative diagram:
\[\xymatrix{
& R\ar@{(>->}[rrr]\ar[dd]^(.65){p_R}|(.49)\hole & & & R+R^c\ar@{(>->}[dl]^r\ar@{->>}[dd]|\hole^(.65){p_{R+R^c}} & & & R^c\ar@{(>->}[lll]\ar@{->>}[dd]^{p_{R^c}} \\
R\ar@{}@<-1.45ex>[ur]^[@!43]{=}\ar@{(>->}[rrr]^(.6){p_X^{-1}(m)}\ar@{->>}[dd] & & & X\ar@{->>}[dd]^(.65){p_X} & & & R^c\ar@{(>->}[lll]_(.4){p_X^{-1}(m^c)}\ar@{}@<-1.45ex>[ur]^[@!43]{=}\ar@{->>}[dd] & \\
& \Pi R\ar@{(>->}[rrr]|(.54)\hole & & & \Pi R+\Pi R^c\ar@{->>}[dl]^{\Pi r} & & & \Pi R^c\ar@{(>->}[lll]|(.34)\hole \\
B\ar@{}@<-1.45ex>[ur]^[@!43]{\cong}\ar@{(>->}[rrr]_m & & & \Pi X & & & B^c\ar@{(>->}[lll]^{m^c}\ar@{}@<-1.45ex>[ur]^[@!43]{\cong} &
}\]
Therefore, as $\Pi R^c\cong B^c$ and accordingly $\Pi R+\Pi R^c\cong\Pi R\cup\Pi R^c$, then $\Pi r$ is monic. So $B$ is complemented, and thus $\dec(\E)$ is a topos with $2$ as its subobject classifier.

Conversely, suppose $\dec(\E)$ is a topos, then it must be boolean (see Acu\~{n}a Ortega and Linton\cite[Observation 2.6]{MR0555540}). Let $f:X\to Y$ be $\neg\neg$-dense, since the composition of $\neg\neg$-dense arrows is $\neg\neg$-dense, it follows that $\Pi(f)\circ p_X=p_Y\circ f$ is dense. Hence it must be epic. 
\endproof

The following result shows that the property of having pneumoconnected fibers is also present in the canonical map from an object to its sheafification. 
\begin{cor}\label{T:negnegsep}
Let $\E$ be topos satisfying NS and let $m_{\neg\neg}^X:X\twoheadrightarrow M_{\neg\neg}X$  be the reflector of the inclusion of category of $\neg\neg$-separated objects of $\E$. Then $m_{\neg\neg}$  has pneumoconnected fibers. Consequently, the sheafification functor also has pneumoconnected fibers\footnote{In the presence of NS, this suggests there ought to be a description that characterizes the behavior of the fibers of $m^X_j:X\rightarrow M_jX$ for an arbitrary local operator $j$, which might then provide a description for the required behavior of the fibers of the unit of $f^*f_!$ of an arbitrary precohesion $f$. Nothing thus far eases the work required to syntactically describe the image of $f^*$.}. 
\end{cor}
\begin{proof} Immediate from \ref{T:Lemon}, since every decidable object is $\neg\neg$-separated.
\end{proof}

 \hypertarget{LemonProof}
The proof of \ref{T:Lemon} is split into the next two results. 
\begin{proposicion}\label{L:(i)->(ii)}
Let $\E$ be a topos satisfying NS. Let $q:X\rightarrow Q$ be such that for every arrow $f:X\rightarrow 2$ there is an arrow $f':Q\rightarrow 2$ making the following diagram commutative:
\[\xymatrix{
X\ar[dr]^f\ar[d]_q & \\
Q\ar[r]_{f'} & 2.
}\]
Then $q$ has pneumoconnected fibers.
\end{proposicion}
\begin{proof}
Define
\[
R:=\{z\in Q:\exists v\in P_c(X)((q^{-1}(z)\cap v\neq\varnothing)\wedge(q^{-1}(z)\cap v^c\neq\varnothing))\}.
\]
Suppose for contradiction that $R$ is not initial. Then, there is a point $a:1\to Q$ and a complemented $A\rightarrowtail X$ such that
\begin{equation}
q^{-1}(a)\cap A\neq\varnothing\wedge q^{-1}(a)\cap A^c\neq\varnothing.
\end{equation}

For $A$ corresponds an arrow $\xi:X\to 2$. Let $\xi':Q\to 2$ be such that $\xi=\xi' q$. That means that $\xi' a$ must be both $0!$ and $1!$. A contradiction.  Therefore, $R$ is initial.  This means that 
\[
\forall v\in P_c(X)\neg((q^{-1}(z)\cap v\neq\varnothing)\wedge(q^{-1}(z)\cap v^c\neq\varnothing)
\]
\pagebreak 

\noindent is universally valid for $z\in Q$, which in turn means that
\[
\neg\neg((q^{-1}(z)\cap v=\varnothing)\vee(q^{-1}(z)\cap v^c=\varnothing))
\]
is universally valid for $z\in Q$ and $v\in P_c(X)$, as promised. 
\end{proof}
\begin{proposicion}\label{L:(ii)->(iii)}
Let $\E$ be a topos satisfying NS and $q:X\twoheadrightarrow Q$ be epic with pneumoconnected fibers. Every arrow $X\rightarrow Y$ with $Y$ decidable factors through $q$.
\end{proposicion}
\begin{proof}

Given an arrow $f:X\rightarrow Y$ with $Y$ decidable,  let $\lvert f\rvert\rightarrowtail X\times Y$ and $\lvert q\rvert\rightarrowtail X\times Q$ be the graphs of $f$ and $q$ respectively. The goal is to find $f'$ through it graph. Consider the subobject  
\[
G=\{\langle z,y\rangle:\exists x(\langle x,z\rangle\in\lvert q\rvert\wedge\langle x,y\rangle\in\lvert f\rvert)\}\rightarrowtail Q\times Y.
\]
Since $q$ is epic,
\begin{equation}\label{E:pneumoepic}
\exists y(\langle z,y\rangle\in G).
\end{equation}
is universally valid for $z\in Q$.  To see that $G$ is indeed the graph of a function $f'$---as per \eqref{E:graph}---, what remains to verify is uniqueness in $y$.  Let 
\[
R:=\{z\in Q: \exists \langle y, y'\rangle\in \Delta_Y^c.\langle z, y\rangle\in G\wedge \langle z, y'\rangle\in G \}
\]
Assume for contradiction that $R$ is not initial. Then there are points $a:1\to Q$, $b,c:1\to Y$ such that $b$ is distinct from $c$ and such that $\langle a, b\rangle$ and $\langle a, c\rangle$ factor through $G$. That means that there exists points $d:1\to q^{-1}(a)\cap f^{-1}(b)$ and $e:1\to q^{-1}(a)\cap f^{-1}(c)$.  

Since $Y$ is decidable, $b$ complemented and thus so is $f^{-1}(b)$.  As $f^{-1}(c)$ is a subobject of $f^{-1}(b)^c$, $(q^{-1}(a)\cap(f^{-1}(b))^c\neq\varnothing)$.  This means that $q^{-1}(a)$ would not be connected. A contradiction. 

Therefore, by NS, $R$ must be initial.  Thus,
 \[
\neg\exists \langle y, y'\rangle\in \Delta_Y^c.\langle z, y\rangle\in G\wedge \langle z, y'\rangle\in G
\]
is universally valid for $z\in Q$. Wherefrom,
\[
\langle y, y'\rangle\in \Delta_Y^c\Rightarrow\neg(\langle z, y\rangle\in G\wedge \langle z, y'\rangle\in G)
\]
\pagebreak

\noindent is universally valid for $z\in Q$ and $\langle y, y'\rangle\in Y\times Y$. By contrapositive, 
\[
(\langle z, y\rangle\in G\wedge \langle z, y'\rangle\in G)\Rightarrow \neg\neg(y=y')
\]
is universally valid for $z\in Q$ and $\langle y, y'\rangle\in Y\times Y$, since $\alpha$ always implies $\neg\neg\alpha$. But using the decidability once more, 
\[
(\langle z, y\rangle\in G\wedge \langle z, y'\rangle\in G)\Rightarrow y=y'
\]
is universally valid for $z\in Q$ and $\langle y, y'\rangle\in Y\times Y$. Which yields uniqueness.  Therefore, $G$ is the graph of an arrow $f':Q\to Y$ that factors $f$.
 \end{proof}
By delving  deeper into the internal logic of the topos it is possible to do away with NS in the previous proof, yet this would go beyond the present purposes.

\begin{proof}[Proof of the Fiber Pneumoconnectedness Lemma \ref{T:Lemon}] 
\ref{L:(i)->(ii)} yields \eqref{I:decquotfact2}$\Rightarrow$\eqref{I:PCfibers}, \linebreak \ref{L:(ii)->(iii)} yields \eqref{I:PCfibers}$\Rightarrow$\eqref{I:decquotfactdec}. Finally \eqref{I:decquotfactdec}$\Rightarrow$\eqref{I:decquotfact2} is trivial since $2$ is decidable.
\end{proof}


\section{Precohesiveness}\label{S:preco}
Recall that a topos $\E$ is precohesive over a topos $\mathcal S$ if there is a string of adjunctions
\begin{equation}\label{E:preco}
f_!\dashv f^*\dashv f_*\dashv f^!:\E\rightarrow\mathcal S
\end{equation}
such that $f^*$ is fully faithful, $f_!$ preserves finite products, and that the counit $f^*f_*\rightarrow 1$ is monic (See Lawvere and Menni \cite[Lemma 3.2]{MR3365705}).

From  \hyperlink{T:TeorA}{Theorem A} and from the results in \cite{MR0877866} it is evident that NS + DSO + DQO yields that $\E$ is precohesive over $\dec(\E)$.

To provide a converse in the presence of NS, let $f$ be as in \eqref{E:preco} over a boolean base.  It is proved in \ref{T:DecReflexImpliesDQO} that DQO holds on $\E$.  

Since this means in particular that the unit $\sigma:1\to f^\ast f_!$ is epic, by \cite[Proposition 2.2]{MR3365705}) this is equivalent to the counit $\beta:f^\ast f_\ast\rightarrow 1$ being monic. To verify the uniqueness in DSO, let $g:A\rightarrowtail X$  with $A$ decidable have the same factoring property for arrows from $1$. Then there is a unique arrow $g':A\rightarrowtail f^\ast f_\ast X$ such that the following diagram commutes:
\[\xymatrix{
A\ar@{(>->}[dr]^g\ar@{(>->}[d]_{g'} & \\
f^\ast f_\ast X\ar@{(>->}[r]_(.6){\beta_X} & X
}\]
It remains to see that $g'$ is epic.  By virtue of \ref{T:TeorD}, it suffices to verify that it is dense, since then 
\[\xymatrix{
A\ar[r]^{g'}\ar[d]_\cong^{\sigma_A} & f_\ast X\ar[d]_{\sigma_{f_\ast X}}^\cong \\
f_{!}A\ar@{->>}[r]_{f_{!}g'} & f_{!}f_\ast X
}\]
$f_!g'$ is epic and thus so is $g'$. 

To this effect, notice that by NS, exactly one of the following is universally valid: $f^\ast f_\ast X\cap(A^{\neg\neg})^c=\varnothing$ or $\neg(f^\ast f_\ast X\cap(A^{\neg\neg})^c=\varnothing)$, since neither has free variables and are thus interpreted as points in $\Omega$. Assuming for contradiction the latter holds, there exists a global element $a:1\to f^\ast  f_\ast X\cap(A^{\neg\neg})^c$. But since $f^\ast f_\ast X\cap(A^{\neg\neg})^c\rightarrowtail X$, $a$ factors through $A$. That is,
\[
a\in A\wedge\neg(a\in A)
\]
would hold---a contradiction. Therefore,  $f_\ast X\cap(A^{\neg\neg})^c=\varnothing$ holds. Whence $A$ is $\neg\neg$-dense in $f_\ast X$. This finishes the proof of \hyperlink{T:TeorC}{Theorem C}.\qed


\end{document}